\documentclass{birkjour}
\usepackage{graphicx}
\usepackage{graphics}
\usepackage{amssymb, amsmath, amsfonts}

 \newtheorem{theorem}{Theorem}[section]
 \newtheorem{lemma}[theorem]{Lemma}
 \theoremstyle{definition}
 \newtheorem{definition}[theorem]{Definition}
 \theoremstyle{remark}
 \newtheorem{remark}[theorem]{Remark}
 \newtheorem*{example}{Example}
 \numberwithin{equation}{section}

\newcommand{\toto}{\rightrightarrows}
\newcommand{\R}{{\mathbb R}}
\newcommand{\N}{{\mathbb N}}

\newcommand{\To}{\longrightarrow}

\def\1{\^{\i}}
\def\2{\u{a}}
\def\3{\c{s}}
\def\4{\^{a}}
\def\5{\c{t}}
\def\a{\alpha}

\def\e{\epsilon}

\def\l{\lambda}
\def\<{\langle}
\def\>{\rangle}

\DeclareMathOperator*\cl{cl}
\DeclareMathOperator*\co{co}

\DeclareMathOperator*\inte{int}

\begin{document}

%-------------------------------------------------------------------------
% editorial commands: to be inserted by the editorial office
%
%\firstpage{1} \volume{228} \Copyrightyear{2004} \DOI{003-0001}
%\seriesextraline{This is the Concrete Title of this Book\br H.E. R and S.T.C. W, Eds.}
%\firstpage{1}

\title[Vector  Equilibrium Problems on Dense Sets]
 {Vector  Equilibrium Problems on Dense Sets}

%----------Author 1
\author{Szil\'ard L\' aszl\' o}

\address{Department of Mathematics\\ Technical University of Cluj-Napoca\\
              Str. Memorandumului nr. 28, 400114 Cluj-Napoca, Romania.}
              \email{laszlosziszi@yahoo.com}

\thanks{This work was supported by a grant of the Romanian Ministry of Education, CNCS - UEFISCDI, project number PN-II-RU-PD-2012-3 -0166.}



%\date{January 1, 2004}
\begin{abstract}
In this paper we provide sufficient conditions that ensure the existence of the solution of some vector equilibrium  problems in Hausdorff topological vector spaces ordered by a cone.  The conditions that we consider are imposed not on the whole domain of the operators involved, but rather on a self segment-dense subset of it, a special type of dense subset. We apply the results obtained to vector optimization and vector variational inequalities.
\end{abstract}
%%% ----------------------------------------------------------------------
\maketitle
%%% ----------------------------------------------------------------------
%\tableofcontents
\section{Introduction}
Equilibrium problems play an important role in nonlinear analysis especially because they provide a unified framework for treating optimization problems, fixed points, saddle points as well as many important problems in physics and mathematical economics, such as location problems or Nash equilibria in game theory. The foundation of (scalar) equilibrium theory has been laid down by Ky Fan \cite{Fan1}, his minimax inequality still being considered one of the most notable results in this field.
 The classical scalar equilibrium problem  \cite{Fan1}, described by a bifunction %
$\varphi :K\times K \longrightarrow {\mathbb{R}}$, consists in finding $x_0\in K$ such
that
\begin{equation*}
\varphi(x_0,y)\ge 0,\,\forall y\in K.
\end{equation*}

We recall the famous existence result of Ky Fan.

\begin{theorem}
\label{tKF} Let $K$ be a nonempty, convex and compact subset of the Hausdorff topological vector space $X$ and let $\varphi
:K \times K \longrightarrow{\mathbb{R}}$ be a bifunction satisfying

\begin{itemize}
\item[(i)] $\forall y\in K$, the function $x\longrightarrow \varphi(x,y)$ is
upper semicontinuous on $K,$

\item[(ii)] $\forall x\in K,$ the function $y\to\varphi(x,y)$ is quasiconvex on $%
K,$

\item[(iii)] $\forall x\in K,\, \varphi(x, x)\ge 0.$

Then, there exists an element $x_0\in K$ such that
\begin{equation*}
\varphi(x_0,y)\ge 0,\,\forall y\in K.
\end{equation*}
\end{itemize}
\end{theorem}

Starting with the pioneering work of Giannessi \cite{G1}, several extensions of the scalar equilibrium problem  to the vector case have been considered. These vector equilibrium problems, much like their scalar counterpart, offer a unified framework for treating vector optimization, vector variational inequalities or cone saddle point problems, to name just a few \cite{A,AKY,AKY1,AOS,Go1,GRTZ}.

Let $X$ and  $Z$ be  locally convex Hausdorff topological vector spaces, let $K\subseteq X$ be a nonempty set and let $C\subseteq Z$ be a  convex and pointed cone.
Assume that the interior of the cone $C$, denoted by $\inte C$, is nonempty and consider the mapping $f:K\times K\To Z.$  The vector equilibrium problem, introduced in \cite{AOS}, consist in finding $x_0\in K$, such that
\begin{equation}\label{p1}
f(x_0,y)\not\in-\inte C,\,\forall y\in K.
\end{equation}
Recall that this problem is called weak vector equilibrium problem \cite{G,G1}.

The following equilibrium problem, called strong vector equilibrium problem \cite{G,G1}, is also a valid extension of the scalar equilibrium problem to vector valued case. In the formulation of the strong vector equilibrium problem we do not assume the nonemptyness of the  interior of the cone $C.$

The strong vector equilibrium problem consists in finding $x_0\in K$, such that
\begin{equation}\label{p2}
f(x_0,y)\not\in- C\setminus\{0\},\,\forall y\in K.
\end{equation}

It can easily be observed, that for $Z=\R$ and $C=\R_+=[0,\infty),$ the previous problems reduce to the classical scalar equilibrium  problem. % that is, find $x_0\in K$ such that $f(x_0,y)\ge 0$ for all $y\in K.$
 Note that, if $\inte C\neq\emptyset$ and $x_0\in K$ is a solution of (\ref{p2}), then $x_0$ is also a solution of (\ref{p1}).

In this paper, we obtain some existence results of the solution for the vector equilibrium problem (\ref{p1}), respectively (\ref{p2}). The conditions that we consider are imposed on a  special type of dense subset of $K$, that we call self-segment-dense \cite{LaVi1}. The notion of a self-segment-dense set has been successfully used in the context of scalar and set-valued equilibrium problems  and generalized set-valued monotone operators in \cite{LaVi, LaVi1}. Moreover, we found very useful applications of this notion in economics (a Debreu-Gale-Nikaido-type theorem) and game theory (existence of Nash equilibria).

The paper is organized as follows. In next section, we introduce some preliminary notions and the necessary apparatus that we need in order to obtain our results. In section 3 and section 4 we state our results concerning on weak, respectively strong vector equilibrium problems. Our conditions, which ensure the solution existence of the above mentioned vector equilibrium problems, are  considerably weakening the existing conditions from the literature. More precisely, we assume some continuity, respectively convexity properties of the vector bifunction $f$ involved in problem (\ref{p1}), respectively(\ref{p2}), not on the whole set $K$, but rather on a self-segment-dense subset of it.  Also the diagonal property $f(x,x)\not\in-\inte C,$ respectively $f(x,x)\notin-C\setminus\{0\}$ is assumed on a self-segment-dense set $D\subseteq K$ only. We pay a special attention to the case when $K$ is not necessarily compact but closed, in particular, when $K$ is a closed subset of a reflexive Banach  space. We show that these results fail, if we replace the self-segment-denseness of $D$ by its denseness. Finally, we apply our results to vector optimization and vector variational inequalities.

\section{Preliminaries}

Let $X$ be a real Hausdorff, locally convex topological vector space.  For a non-empty set $D\subseteq X$, we denote by $\inte D$ its interior, by $\cl D$ its closure and by $\mbox{span} D$ the subspace of $X$ generated by $D$.  We say that $P\subseteq D$ is dense in $D$ iff $D\subseteq \cl P$, and that $P\subseteq X$ is closed with respect to $D$ iff $\cl P \cap D=P\cap D.$ Recall that a set $C\subseteq X$ is a cone, iff $tk\in C$ for all $c\in C$   and $t\ge 0.$ The cone $C$ is convex if $C+C=C,$ and pointed if $C\cap (-C)=\{0\}.$ Note that a cone $C$ induce a partial ordering on $Z$, that is $z_1\le z_2\Leftrightarrow z_2-z_1\in C.$ In the sequel when we use $\inte C,$  we tacitly assume that the cone $C$ has nonempty interior. Following the same logical approach, one can introduce the strict inequality  $z_1< z_2\Leftrightarrow z_2-z_1\in \inte C,$ or $z_1< z_2\Leftrightarrow z_2-z_1\in  C\setminus\{0\}.$ These relations lead to  saying, that $z_1\not< z_2\Leftrightarrow z_2-z_1\not\in-\inte C$, or $z_1\not< z_2\Leftrightarrow z_2-z_1\not\in-C\setminus\{0\}.$ It is an easy exercise to show that $C+C\setminus\{0\}=C\setminus\{0\},$ respectively  $\inte C+C= \inte C.$

Let   $Z$ be  another  Hausdorff, locally convex topological vector space, let $K\subseteq X$ be a nonempty set and let $C\subseteq Z$ be a  convex and pointed cone.

%Let us introduce the dual cone of $C$, that is
%$C^*=\{z^*\in Z^*:\<z^*,z\>\ge 0,\mbox{ for all }z\in C\},$ where $Z^*$ denotes the topological dual space of $Z$ and $\<z^*,z\>$ denotes the scalar %$z^*(z)\in\R.$ The dual cone of $C^*$ is $C^{**}=C.$

%It is well known that for $z^*\in C^*\setminus\{0\}$ one has $z^*(z)>0$ for all $z\in\inte C.$

A map $f:K\To Z$ is said to be C-upper semicontinuous
at $x\in K$ iff for any neighborhood $V$ of $f(x)$ there exists a neighborhood $U$ of $x$ such that $f(u)\in  V-C$ for all $u\in U\cap K$. Obviously, if $f$ is continuous at $x\in K,$ then it is  also C-upper semicontinuous at $x\in K$.
Assume that $C$ has nonempty interior.  According to \cite{Ta} $f$ is C-upper semicontinuous at $x\in K,$ if and only if, for any $k\in\inte C$, there exists a neighborhood $U$ of $x$ such that
$f(u) \in f(x) + k -\inte C$ for all $u \in U\cap K.$

 Similarly, even if $\inte C=\emptyset$, one can introduce the so called strongly C-upper semicontinuity of $f$ as follows: $f$ is strongly C-upper semicontinuous at $x\in K,$ if and only if, for any $k\in C\setminus\{0\}$, there exists a neighborhood $U$ of $x$ such that
$f(u) \in f(x) + k - C\setminus\{0\}$ for all $u \in U\cap K.$  The map $f:K\To Z$ is said to be C-lower semicontinuous, respectively strongly C-lower semicontinuous at $x\in K$ iff the map $-f$ is C-upper semicontinuous, respectively strongly C-upper semicontinuous at $x\in K.$

\begin{definition}
The function $f : K\to Z$ is called $C$-convex on $K$, iff  for all $x_1,x_2,\ldots, x_n\in K$, $n\in \N$ and $\lambda_i\ge 0,\, i\in\{1,2,\ldots,n\},$ with
$\sum_{i=1}^n \lambda_i=1,$ such that $\sum_{i=1}^n \lambda_i x_i\in K,$ one
has
$$f\left(\sum_{i=1}^n \l_ix_i\right) \le \sum_{i=1}^n \l_i f(x_i).$$
$f$ is said to be $C$-concave on $K$, iff $-f$ is $C$-convex on $K$.
\end{definition}

Note that in these definitions we do not assume that $K$ is convex. We will  use the following notations for the open, respectively
closed, line segments in $X$ with the endpoints $x$ and $y$
\begin{eqnarray*}
]x,y[ &:=&\big\{z\in X:z=x+t(y-x),\,t\in ]0,1[\big\}, \\
\lbrack x,y] &:=&\big\{z\in X:z=x+t(y-x),\,t\in \lbrack 0,1]\big\}.
\end{eqnarray*}
The line segments $]x,y],$ respectively $[x,y[$ are defined similarly.
In \cite{DTL}, Definition 3.4, The Luc has introduced the notion of a so-called \emph{segment-dense} set. Let $V\subseteq X$ be a convex set. One
says that the set $U\subseteq V$ is segment-dense in $V$ if for each $x\in V$ there can be found $y\in U$ such that $x$ is a cluster point of the set $%
[x,y]\cap U.$

In what follows we present a denseness notion (cf. \cite{LaVi,LaVi1}) which is slightly different from the concept of The Luc presented above, but which is better suited for our needs.

\begin{definition}
\label{dd} Consider the sets $U\subseteq V\subseteq X$ and assume that $V$
is convex.
We say that $U$ is self segment-dense in $V$ if $U$ is dense in $V$ and
\begin{equation*}
\forall x,y\in U,\mbox{  the set }\lbrack x,y]\cap U\mbox{  is dense in }%
\lbrack x,y].
\end{equation*}
\end{definition}

\begin{remark}\textrm{ Obviously in one dimension the concepts of a segment-dense set respectively a self
segment-dense set are equivalent to the concept of a dense set. }
\textrm{\ }
\end{remark}

In what follows we provide an essential example of a self segment-dense set.

\begin{example}
\textrm{(\cite{LaVi1}, Example 2.1)\label{EEE} Let $V$ be the two
dimensional Euclidean space ${\mathbb{R}}^2$ and define $U$ to be the set
\begin{equation*}
U :=\{(p,q) \in{\mathbb{R}}^2 : p\in \mathbb{Q},\, q\in\mathbb{Q}\},
\end{equation*}
where $\mathbb{Q}$ denotes the set of all rational numbers. Then, it is
clear that $U$ is dense in ${\mathbb{R}}^2.$ On the other hand $U$ is not
segment-dense in ${\mathbb{R}}^2,$ since for $x=(0, \sqrt{2})\in {\mathbb{R}}%
^2 $ and for every $y=(p,q)\in U$, one has $[x, y] \cap U = \{y\}.$ }

\textrm{It can easily be observed that $U$ is self segment-dense in ${%
\mathbb{R}}^2$, since for every $x,y\in U$ $x=(p,q),\,y=(r,s)$ we have $%
[x,y]\cap U=\{(p+t(r-p),q+t(s-q)): t\in[0,1]\cap\mathbb{Q}\},$ which is
obviously dense in $[x,y].$ }
\end{example}

To further circumscribe the notion of a self segment-dense set we provide an example %
of a subset that is dense but not self segment-dense.

\begin{example}\label{ex1}
\textrm{Let $X$ be an infinite dimensional real Hilbert space, it is known that the unit sphere %
$$\left\{ x\in X:\left\Vert x\right\Vert =1\right\},$$
is dense with respect to the weak topology in the unit ball $\left\{ x\in X:\left\Vert x\right\Vert \leq1 \right\}$, but %
it is obviously not self segment-dense since any segment with endpoints on the sphere does not intersect the sphere in any other points.}
\end{example}

\begin{remark}
\textrm{Note that every dense convex subset of a Banach space is self
segment-dense. In particular dense subspaces and dense affine subsets are
self segment-dense. }
\end{remark}

In subsequent section, the notion of a KKM map and the well-known intersection Lemma due to Ky Fan %
\cite{Fan} will be needed.

\begin{definition}
(Knaster-Kuratowski-Mazurkiewicz) Let $X$ be a Hausdorff topological vector
space and let $M\subseteq X.$ The application $G:M\rightrightarrows X$ is
called a KKM application if for every finite number of elements $%
x_1,x_2,\dots,x_n\in M$ one has $$\co\{x_1,x_2,\ldots,x_n\}\subseteq %
\displaystyle\bigcup_{i=1}^n G(x_i).$$
\end{definition}

\begin{lemma}
Let $X$ be a Hausdorff topological vector space, $M\subseteq X$ and $%
G:M\rightrightarrows X$ be a KKM application. If $G(x)$ is closed for every $%
x\in M$, and there exists $x_{0}\in M,$ such that $G(x_{0})$ is compact,
then
\begin{equation*}
\bigcap_{x\in M}G(x)\neq \emptyset .
\end{equation*}
\end{lemma}

The proof of the results obtained in the next sections are based on the following lemma from \cite{LaVi1}.
\begin{lemma}
\label{l31} Let $X$ be a Hausdorff locally convex topological vector space, let $V\subseteq X$ be a convex set and let $U\subseteq V$ a self-segment-dense set in $V.$ Then for all finite subset $$\{u_1,u_2,\ldots,u_n\}\subseteq U\mbox{ one has }$$
$$\cl(\co\{u_1,u_2,\ldots,u_n\}\cap U)=\co%
\{u_1,u_2,\ldots,u_n\}.$$
\end{lemma}

Let us emphasize that this result does not remain valid in case we replace the self-segment-denseness of $U$ in $V$, by its denseness in $V,$ as the next example shows.
\begin{example} Let $V$ be the closed unit ball of an infinite dimensional Banach space $X$, and let $x,y\in V, \, x\neq y.$ Moreover, consider $u,v\in ]x,y[,$ $u=x+t_1(y-x),$ $v=x+t_2(y-x),$ with $t_1,t_2\in]0,1[,\,t_1<t_2.$ Then obviously $U=V\setminus [u,v]$ is dense in $V$, but not self-segment-dense, since for $x,y\in U$ the set  $[x,y]\cap U=[x,u[\,\cup\, ]v,y]$ is not dense in $[x,y].$ This also shows, that $\cl(\co\{x,y\}\cap U)\neq\co\{x,y\}.$
\end{example}

%%%%%%%%%%%%%%%%%%%%%%%%%%%%%%%%%%%%%%%%%%%%%%%%%%%%%%%%%%%%%%%%%%%%

\section{Self Segment-Dense Sets and the Weak Vector Equilibrium Problem}

%In order to obtain existence results of the solutions for the vector equilibrium problems (\ref{p1}), (\ref{p2}) respectively (\ref{p3}), we will study the scalar equilibrium problem (\ref{sp1}).
%Let $K\subseteq X$ be a nonempty set and let $C\subseteq Z$ be a convex and  pointed cone. %We treat the scalar equilibrium problem of finding $a_0\in A$ such that $\varphi(a_0,b)\ge 0,\,\forall b\in B,$ where $\varphi:A\times B\To\R$ and $A=C\setminus\{0\}\times K,\, B=K$, respectively $A=K,\, B=C\times K$.
In this section, we obtain some sufficient conditions that ensure the existence of a solution for the weak vector equilibrium problem (\ref{p1}). The conditions, that we consider, are imposed not on the whole domain of the vector bifunction $f$, but rather on a self-segment-dense subset of it. We also show, that the  self-segment-dense property of this set is essential in obtaining our results, and cannot be replaced by the usually denseness property.  Lemma \ref{l31} plays an important role in the proofs of our results. We treat both the cases when the set $K$, the domain of the vector bifunction $f$, is compact, respectively closed. We pay a special attention to the case when the $K$ is a closed subset of reflexive Banach space $X.$
The main result of this section is the following.

\begin{theorem}
\label{t1} Let $X$ and $Z$ be  Hausdorff, locally convex topological vector spaces,  let $C\subseteq Z$ be a convex,  pointed cone with nonempty interior and
let $K$ be a nonempty, convex and compact subset of $X$. Let $D\subseteq K$ be a
self segment-dense set and consider the mapping $f :K \times K \longrightarrow{Z}$  satisfying

\begin{itemize}
\item[(i)] $\forall y\in D,$ the mapping $x\longrightarrow f(x,y)$ is
C-upper semicontinuous on $K,$

\item[(ii)] $\forall x\in K,$ the mapping $y\longrightarrow f(x, y)$ is C-upper
semicontinuous on $K\setminus D$,

\item[(iii)] $\forall x\in D,$ the mapping $y\longrightarrow f(x,y)$ is C-convex on $%
D,$

\item[(iv)] $\forall x\in D,\, f(x, x)\not\in-\inte  C.$
\end{itemize}

Then, there exists an element $x_0\in K$ such that
\begin{equation*}
f(x_0,y)\not\in -\inte{C},\,\forall y\in K.
\end{equation*}
\end{theorem}

\begin{proof} We give two different proofs.

\textbf{I.} Assume the contrary, that is, for every $x\in K$ there exists $y\in K$ such that $f(x,y)\in-\inte C.$ Then, for every $y\in K$ consider $V_y=\{x\in K: f(x,y)\in-\inte C\}.$ It is obvious that $K\subseteq \cup_{y\in K} V_y.$ We show that $(V_y)_{y\in D}$ is an open cover of $K.$ First of all observe that for all $y\in D,$ one has  $V_y=K\setminus G(y)$, where $G(y)$ is the set $\{x\in K:f(x,y)\not\in-\inte C\}.$ We show that $G(y)$ is closed for all $y\in D.$ Indeed, for fixed $y_0\in D$ consider the net $(x_\a)\subseteq G(y_0)$ and let $\lim x_\a=x_0.$ Assume that $x_0\not\in G(y_0).$ Then $f(x_0,y_0)\in-\inte C$.
 According to (i) the function $x\longrightarrow f(x,y_0)$ is
C-upper semicontinuous at $x_0$, hence  for every $k\in\inte C$ there exists $U,$ a neighborhood of $x_0,$ such that $f(x,y_0)\in f(x_0,y_0)+k-\inte C$ for all $x\in U.$ But then for $k=-f(x_0,y_0)\in\inte C$  one obtains that there exits $\a_0$ such that $f(x_\a,y_0)\in -\inte C,$ for $\a\ge\a_0,$ which contradicts the fact that $(x_\a)\subseteq G(y_0)$. %for all $y\in D$, where $G$ is the set-valued map defined in the first part of the proof. But we have shown that $G(y)$ is closed for all $y\in D$,
Consequently, $V_y$ is open for every $y\in D.$

Assume now that there exists $x_0\in K$ such that $x_0\not\in \cup_{y\in D} V_y.$ Then $f(x_0,y)\not\in-\inte C$, for all $y\in D.$ We show that  $f(x_0,y)\not\in-\inte C$, for all $y\in K.$ Indeed, for $y_0\in K\setminus D,$ by the denseness of $D$ in $K,$ we have that  for every neighborhood $U$ of $y_0$ there exists a $u_0\in U\cap D.$  At this point, the assumption $(ii)$, the upper semicontinuity of $f(x_0, y)$ on $K\setminus D$, assures that for all $k\in\inte C$ there exists a neighborhood $U$ of $y_0$ such that $ f(x_0,u)\in f(x_0,y_0)+k-\inte C.$ Assume that $f(x_0,y_0)\in -\inte C.$ Then let $k=-f(x_0,y_0)\in\inte C.$ Thus, we have that there exists a neighborhood $U$ of $y_0,$ such that $ f(x_0,u)\in -\inte C.$ But, by choosing $u_0\in U\cap D$ we get that $f(x_0,u_0)\in -\inte C,$ contradiction. Hence, $f(x_0,y)\not\in-\inte C$, for all $y\in K$, which contradicts our assumption, that for every $x\in K$ there exists $y\in K,$ such that $f(x,y)\in-\inte C.$

Consequently, $(V_y)_{y\in D}$ is an open cover of the compact set $K$, in conclusion it contains a finite subcover. In other words, there exists $y_1,y_2,...,y_n\in D$ such that $K\subseteq \cup_{i=1}^n V_{y_i}.$ Consider $\big(p_i\big)_{i=\overline{1,n}}$  a continuous partition of  unity associated to the open cover $\big(V_{y_i}\big)_{i=\overline{1,n}}$. Then $p_i:K\To[0,1]$ is continuous and $\mbox{supp}(p_i)=\cl\{x\in K: p_i(x)\neq 0\}\subseteq V_{y_i}$ for all $i\in\{1,2,...,n\}$, moreover $\sum_{i=1}^n p_i(x)=1,$ for all $x\in K.$

Consider the mapping $\varphi:\co\{y_1,y_2,...,y_n\}\To\co\{y_1,y_2,...,y_n\},$
$$\varphi(x)=\sum_{i=1}^n p_i(x)y_i.$$
Obviously $\varphi$ is continuous, and $\co\{y_1,y_2,...,y_n\}$ is a compact and convex subset of the  finite dimensional  space $\mbox{span}\{y_1,y_2,...,y_n\}.$ Hence, by the Brouwer fixed point theorem, there exists $x_0\in \co\{y_1,y_2,...,y_n\}$ such that $\varphi(x_0)=x_0.$

Let $J=\{i\in\{1,2,...,n\}:p_i(x_0)>0\}.$ Obviously $J$ is nonempty, since $\sum_{i\in J} p_i(x_0)=1,$ and $$\varphi(x_0)=\sum_{i=1}^n p_i(x_0)y_i=\sum_{i\in J}p_i(x_0)y_i=x_0.$$
The latter  equality shows, that $x_0\in\co\{y_i:i\in J\}.$ On the other hand, from $p_i(x_0)>0$ for all $i\in J$ we obtain that $x_0\in\cap_{i\in J}V_{y_i}.$ Since $\cap_{i\in J}V_{y_i}$ is open, we  obtain  that there exists $U$ an open and convex neighbourhood of $x_0$ such that $U\subseteq \cap_{i\in J}V_{y_i}.$ Here we find very useful the conclusion of Lemma \ref{l31}. Indeed, since $\co\{y_i:i\in J\}\cap U\neq\emptyset$, according to Lemma \ref{l31}, we have that there exists $y_0\in \co\{y_i:i\in J\}\cap U\cap D.$
Hence, we  have $y_0=\sum_{i\in J}\l_i y_i\in \co\{y_i:i\in J\}\cap U\cap D$, where $\l_i\ge0$ for all $i\in J$ and $\sum_{i\in J}\l_i=1$. By (iv), in the hypothesis of the theorem, one gets $f(y_0,y_0)\not\in-\inte C.$ On the other hand, by using (iii) we get $$f(y_0,y_0) =f(y_0,\sum_{i\in J}\l_i y_i)\le\sum_{i\in J}\l_i f(y_0,y_i),$$ which shows that $\sum_{i\in J}\l_i f(y_0,y_i)-f(y_0,y_0)\in C.$ But, $y_0\in U$, thus $f(y_0,y_i)\in-\inte C,$ for all $i\in J.$ Hence $\sum_{i\in J}\l_i f(y_0,y_i)\in -\inte C,$ which leads to $$-f(y_0,y_0)\in C-\sum_{i\in J}\l_i f(y_0,y_i)\subseteq\inte C,$$ contradiction.
\qed

\textbf{II.} The second proof is based on Ky Fan's Lemma. We consider the set-valued map $$G:D\toto K,\,\, G(y)=\{x\in K:f(x,y)\not\in-\inte C\}.$$ We have shown in the first part of the proof, that $G(y)$ is closed for all $y\in D.$ %Indeed, for fixed $y_0\in D$ consider the net $(x_\a)\subseteq G(y_0)$ and let $\lim x_\a=x_0.$ Assume that $x_0\not\in G(y_0).$ Then $f(x_0,y_0)\in-\inte C$.
 %According to (i) the function $x\longrightarrow f(x,y_0)$ is
%C-upper semicontinuous at $x_0$, hence  for every $k\in\inte C$ there exists $U,$ a neighborhood of $x_0,$ such that $f(x,y_0)\in f(x_0,y_0)+k-\inte C$ for all $x\in U.$ But then for $k=-f(x_0,y_0)\in\inte C$  one obtains that there exits $\a_0$ such that $f(x_\a,y_0)\in -\inte C,$ for $\a\ge\a_0,$ which contradicts the fact that $(x_\a)\subseteq G(y_0)$.
%On the other hand $f(x_\a,y_0)\not\in-\inte C$ for every $\a,$ contradiction.
 %Hence $G(y)\subseteq K$ is closed for all $y\in D$.
 Since $K$ is compact, we have that $G(y)\subseteq K,$ is compact for all $y\in D.$
We show next, that $G$ is a  KKM mapping. We claim that for all $y_1,y_2,...,y_n\in D$ one has $\co\{y_1,y_2,...,y_n\}\cap D\subseteq \bigcup_{i=1}^n G(y_i).$ Indeed, assume that there exist $y_1,y_2,...,y_n\in D$ and  $y\in \co\{y_1,y_2,...,y_n\}\cap D,$ such that $y\not\in \bigcup_{i=1}^n G(y_i).$ Hence, there exist $\l_1,\l_2,\ldots,\l_n\ge 0,\,\sum_{i=1}^n\l_i=1$ such that $\sum_{i=1}^n\l_i y_i\in D$ and $$\sum_{i=1}^n\l_i y_i\not\in\bigcup_{i=1}^n G(y_i).$$
This is equivalent with $f(\sum_{i=1}^n\l_i y_i,y_i)\in-\inte C,\,\mbox{for all } i\in\{1,2,\ldots,n\},$ and hence, by the convexity of $-\inte C$ we have that
$$\sum_{i=1}^n \l_if\left(\sum_{i=1}^n\l_i y_i,y_i\right)\in-\inte C.$$
From assumption $(iii)$,  we have that $$\sum_{i=1}^n \l_i f\left(\sum_{i=1}^n\l_i y_i,y_i\right)- f\left(\sum_{i=1}^n\l_i y_i, \sum_{i=1}^n\l_i y_i\right )\in C,$$ or equivalently,  $$f\left(\sum_{i=1}^n\l_i y_i, \sum_{i=1}^n\l_i y_i\right )\in\sum_{i=1}^n \l_i f\left(\sum_{i=1}^n\l_i y_i,y_i\right)- C\subseteq -\inte C,$$  which contradicts (iv).
Consequently, $$\co\{y_1,y_2,\ldots,y_n\}\cap D\subseteq\bigcup_{i=1}^n G(y_i),$$ holds true, and leads to
 $$\cl(\co\{y_1,y_2,\ldots,y_n\}\cap D)\subseteq\cl\left(\bigcup_{i=1}^n G(y_i)\right).$$

Furthermore, since $G(y_i)$ is closed for all $i\in\{1,2,\ldots,n\}$ we have
$$\cl\left(\bigcup_{i=1}^n G(y_i)\right)=\bigcup_{i=1}^n G(y_i).$$
On the other hand, according to Lemma \ref{l31} we have $$\cl(\co\{y_1,y_2,\ldots,y_n\}\cap D)=\co\{y_1,y_2,\ldots,y_n\},$$ hence
 $$\co\{y_1,y_2,\ldots,y_n\}\subseteq\bigcup_{i=1}^n G(y_i).$$ In conclusion $G$ is a KKM map.

Thus, according to Ky Fan's Lemma, $\bigcap_{y\in  D}G(y)\neq\emptyset.$ In other words, there exists $x_0\in K,$ such that $f(x_0,y)\not\in-\inte  C$ for all $y\in D.$

Finally, if $y_0\in K\setminus D,$ by the denseness of $D$ in $K,$ we obtain that  for every neighborhood $U$ of $y_0,$ there exists a $u_0\in U\cap D.$  At this point, the assumption (ii), the C-upper semicontinuity of $f(x_0, y)$ on $K\setminus D$, assures that for all $k\in\inte C$ there exists a neighborhood $U$ of $y_0$ such that $ f(x_0,u)\in f(x_0,y_0)+k-\inte C.$ Assume that $f(x_0,y_0)\in -\inte C.$ Then, let $k=-f(x_0,y_0)\in\inte C.$ Thus, we have that  there exists a neighborhood $U$ of $y_0,$ such that $ f(x_0,u)\in -\inte C.$ But, by choosing $u_0\in U\cap D$ we get that $f(x_0,u_0)\in -\inte C,$ contradiction.
\end{proof}
\begin{remark}
The compactness of the set $K$ in the hypothesis of the previous theorem is rather a strong condition.  The compactness condition can be removed by assuming only the closedness of $K$ but also a so-called coercivity condition. This can be done in the virtue of Fan's Lemma, which do not require the compactness of the set $G(y)$ for every $y\in K,$ but in only one point. Therefore, the following result holds.
\end{remark}

\begin{theorem}
\label{t11} Let $X$ and $Z$ be  Hausdorff, locally convex topological vector spaces,  let $C\subseteq Z$ be a convex,  pointed cone with nonempty interior, and let $K$ be a nonempty, convex and closed subset of $X$. Let $D\subseteq K$ be a
self segment-dense set, and consider the mapping $f :K \times K \longrightarrow{Z}$  satisfying

\begin{itemize}
\item[(i)] $\forall y\in D,$ the mapping $x\longrightarrow f(x,y)$ is
C-upper semicontinuous on $K,$

\item[(ii)] $\forall x\in K,$ the mapping $y\longrightarrow f(x, y)$ is C-upper
semicontinuous on $K\setminus D$,

\item[(iii)] $\forall x\in D,$ the mapping $y\longrightarrow f(x,y)$ is C-convex on $%
D,$

\item[(iv)] $\forall x\in D,\, f(x, x)\not\in-\inte  C,$

\item[(v)] $\exists K_0\subseteq X$ compact and $y_0\in D,$ such that $f(x,y_0)\in-\inte C,$ for all $x\in K\setminus K_0.$
\end{itemize}

Then, there exists an element $x_0\in K$ such that
\begin{equation*}
f(x_0,y)\not\in -\inte{C},\,\forall y\in K.
\end{equation*}
\end{theorem}

\begin{proof} Consider the set-valued map $$G:D\toto K,\,\, G(y)=\{x\in K:f(x,y)\not\in-\inte C\}.$$ According to the proof of Theorem \ref{t1} $G$ is a KKM map, and $G(y)$ is closed for all $y\in D.$ We show that $G(y_0)$ is compact. For this is enough to show that $G(y_0)\subseteq K_0.$ Assume the contrary, that is $g(y_0)\not\subseteq K_0.$ Then, there exits $z\in G(y_0)\setminus K_0.$ This implies that $z\in K\setminus K_0,$ and according to (v) $f(z,y_0)\in-\inte C,$ which contradicts the fact that $z\in G(y_0).$

Hence, $G(y_0)$ is a closed subset of the compact set $K_0$ which shows that $G(y_0)$ is compact. The rest of the proof is similar to the proof  of Theorem \ref{t1}, therefore we omit it.
\end{proof}

\begin{remark} Condition (v) usually is hard to be verified.  However, it is well known that, in a reflexive Banach space $X$, the closed ball, with radius $r>0$, $\overline{B}_r:=\{x\in X:\|x\|\le r\},$ is weakly compact. Therefore, if we endow the reflexive Banach space $X$ with the weak topology, we can take $K_0=\overline{B}_r\cap K$, hence, condition (v) in Theorem \ref{t11} becomes :
$$\exists r>0\mbox{ and }y_0\in D,\mbox{ such that, for all }x\in K\mbox{ satisfying }\|x\|>r,$$
one has that
$$f(x,y_0)\in-\inte C.$$

Furthermore, in this setting, condition (v) in the hypothesis of Theorem \ref{t11} can be weakened by assuming that  there exists $r>0$ such that, for all $x\in K$ satisfying %
$\|x\|>r$, there exists some $y_0\in K$ with $\|y_0\|<\|x\|,$ and for which the  condition
$$f(x,y_0)\in-\inte C$$
holds.
\end{remark}

More precisely, we have the following result.

\begin{theorem}
\label{t12} Let $X$ be a reflexive Banach space and let $Z$ be  a Hausdorff, locally convex topological vector space.  Let $C\subseteq Z$ be a convex,  pointed cone with nonempty interior, and let $K$ be a nonempty, convex and closed subset of $X$. Let $D\subseteq K$ be a
self segment-dense set in the weak topology of $X$, and consider the mapping $f :K \times K \longrightarrow{Z}$  satisfying

\begin{itemize}
\item[(i)] $\forall y\in D,$ the mapping $x\longrightarrow f(x,y)$ is
C-upper semicontinuous on $K,$ in the weak topology of $X$,

\item[(ii)] $\forall x\in K,$ the mapping $y\longrightarrow f(x, y)$ is C-upper
semicontinuous on $K\setminus D$, in the weak topology of $X$,

\item[(iii)] $\forall x\in K,$ the mapping $y\longrightarrow f(x,y)$ is C-convex on $K,$

\item[(iv)] $\forall x\in K,\, f(x, x)=0,$

\item[(v)] $\exists r>0$ such that, for all $x\in K$, $\|x\|>r$, there exists $y_0\in K$ with $\|y_0\|<\|x\|,$ such that
$f(x,y_0)\in-\inte C\cup\{0\}.$
\end{itemize}

Then, there exists an element $x_0\in K$ such that
\begin{equation*}
f(x_0,y)\not\in -\inte{C},\,\forall y\in K.
\end{equation*}
\end{theorem}
\begin{proof} Let $r>0$ such that (v) holds, and let $r_1>r.$ Consider $K_0=K\cap \overline{B}_{r_1}.$ Obviously, $K_0$ is weakly compact, hence, according to Theorem \ref{t1} there exists $x_0\in K_0$ such that $f(x_0,y)\not\in-\inte C,\,\forall y\in K_0.$ We show, that $f(x_0,y)\not\in-\inte C,\,\forall y\in K.$

First we show, that there exists $z_0\in K_0,\,\|z_0\|<r_1$ such that $f(x_0,z_0)=0.$ Indeed, if $\|x_0\|<r_1$ then let $z_0=x_0$ and the conclusion follows from (iv). Assume now, that $\|x_0\|=r_1>r.$ Then, according to (v), we have that there exists $z_0\in K,\,\|z_0\|<\|x_0\|=r_1$ such that $f(x_0,z_0)\in-\inte C\cup\{0\}.$ On the other hand $z_0\in K_0,$ hence $f(x_0,z_0)\not\in-\inte C,$ which leads to $f(x_0,z_0)=0.$

Let $y\in K.$ Then, there exists $\l\in[0,1]$ such that $\l z_0 +(1-\l)y\in K_0,$ consequently $f(x_0,\l z_0 +(1-\l)y)\not\in-\inte C.$ From (iii) we have
$$\l f(x_0,z_0)+(1-\l)f(x_0,y)-f(x_0,\l z_0 +(1-\l)y)\in C$$
or, equivalently
$$(1-\l)f(x_0,y)-f(x_0,\l z_0 +(1-\l)y)\in C.$$
Assume that $f(x_0,y)\in-\inte C.$ Then, $-f(x_0,\l z_0 +(1-\l)y)\in -(1-\l)f(x_0,y)+ C\subseteq \inte C,$ in other words
$$f(x_0,\l z_0 +(1-\l)y)\in-\inte C,$$
contradiction.
\end{proof}

\begin{remark} If one compares  the hypotheses of Theorem \ref{t12} and Theorem \ref{t1}, or Theorem \ref{t11}, observes that the conditions (iii) and (iv) have considerably been changed.  This is due the fact that condition (v) in Theorem \ref{t12} with the assumptions (iii) and (iv) of Theorem \ref{t1} or Theorem \ref{t11} does not assure the existence of a solution for the weak vector equilibrium problem, when $K$ is  closed but not compact.

Our purpose is to overcome this situation by replacing (v) with a condition that assures the existence of a solution under the original assumptions (iii) and (iv). In fact, we show that, if $\forall x\in K,\, y\longrightarrow f(x, y)$ is C-convex on $D,$ respectively $\forall x\in D,\,f(x,x)\not\in-\inte C,$ instead of (iii), respectively (iv) in the previous theorem, then we can replace (v) with:

$\exists r>0,$ such that, for all $x\in K$ satisfying $ \|x\|\le r,$ there exists $y_0\in D$ with $\|y_0\|<r,$  such that $f(x,y_0)\in-\inte C\cup\{0\}.$
\end{remark}

The following result holds.
\begin{theorem}
\label{t13} Let $X$ be a reflexive Banach space and let $Z$ be  a Hausdorff, locally convex topological vector space.  Let $C\subseteq Z$ be a convex,  pointed cone with nonempty interior, and let $K$ be a nonempty, convex and closed subset of $X$. Let $D\subseteq K$ be a
self segment-dense set in the weak topology of $X$, and consider the mapping $f :K \times K \longrightarrow{Z}$  satisfying

\begin{itemize}
\item[(i)] $\forall y\in D,$ the mapping $x\longrightarrow f(x,y)$ is
C-upper semicontinuous on $K,$ in the weak topology of $X$,

\item[(ii)] $\forall x\in K,$ the mapping $y\longrightarrow f(x, y)$ is C-upper
semicontinuous on $K\setminus D$, in the weak topology of $X$,

\item[(iii)] $\forall x\in K,$ the mapping $y\longrightarrow f(x,y)$ is C-convex on $D,$

\item[(iv)] $\forall x\in D,\, f(x, x)\not\in-\inte C,$

\item[(v)] $\exists r>0$ such that, for all $x\in K$, $\|x\|\le r$, there exists $y_0\in D$ with $\|y_0\|<r,$ such that
$f(x,y_0)\in-\inte C\cup\{0\}.$
\end{itemize}

Then, there exists an element $x_0\in K$ such that
\begin{equation*}
f(x_0,y)\not\in -\inte{C},\,\forall y\in K.
\end{equation*}
\end{theorem}
\begin{proof}  Let $r>0$ such that (v) holds, and consider $K_0=K\cap \overline{B}_{r}.$ Obviously, $K_0$ is weakly compact, hence, according to Theorem \ref{t1} there exists $x_0\in K_0$ such that $f(x_0,y)\not\in-\inte C,\,\forall y\in K_0.$ We show, that $f(x_0,y)\not\in-\inte C,\,\forall y\in K.$
According to (v) there exists $z_0\in D$ with $\|z_0\|<r,$ such that $f(x,z_0)\in-\inte C\cup\{0\}.$ On the other hand, $z_0\in K_0,$ hence  $f(x_0,z_0)\not\in-\inte C.$ Consequently $f(x_0,z_0)=0.$
Let $y\in D\setminus K_0.$ Then, in virtue of self-segment-denseness of $D$ in $K$, there exists $\l\in[0,1]$ such that $\l z_0 +(1-\l)y\in D\cap K_0,$ consequently $f(x_0,\l z_0 +(1-\l)y)\not\in-\inte C.$ From (iii) we have
$$\l f(x_0,z_0)+(1-\l)f(x_0,y)-f(x_0,\l z_0 +(1-\l)y)\in C$$
or, equivalently
$$(1-\l)f(x_0,y)-f(x_0,\l z_0 +(1-\l)y)\in C.$$
Assume that $f(x_0,y)\in-\inte C.$ Then, $-f(x_0,\l z_0 +(1-\l)y)\in -(1-\l)f(x_0,y)+ C\subseteq \inte C,$ or, in other words
$$f(x_0,\l z_0 +(1-\l)y)\in-\inte C,$$
contradiction.

Hence, $f(x_0,y)\not\in-\inte C,$ for all $y\in D.$
Finally, if $y\in K\setminus D$ by the denseness of $D$ in $K$  for every neighborhood $U$ of $y$ there exists a $u_0\in U\cap D.$  At this point, the assumption $(ii)$, the C-upper semicontinuity of $y\To f(x_0, y)$ on $K\setminus D$, assures that for all $k\in\inte C$ there exists a neighborhood $U$ of $y$ such that $ f(x_0,u)\in f(x_0,y)+k-\inte C.$ Assume that $f(x_0,y)\in -\inte C.$ Then let $k=-f(x_0,y)\in\inte C.$ Thus we have that  exists a neighborhood $U$ of $y$ such that $ f(x_0,u)\in -\inte C.$ But, by choosing $u_0\in U\cap D$ we get that $f(x_0,u_0)\in -\inte C,$ contradiction.
\end{proof}

In what follows, we show that the assumption that $D$ is self-segment-dense, in the hypotheses %
of the previous theorems, is essential and it cannot be replaced by the denseness of $D.$ Indeed, let us consider the Hilbert space of square-summable sequences $l_2$, and let $K=\{x\in l_2:\|x\|\le 1\}$ be its closed unit ball, while $D=\{x\in l_2:\|x\|= 1\}$ is the unit sphere. It is well known that $l_2,$ endowed with the weak topology, is a Hausdorff, locally convex topological vector space, and by Banach-Alaoglu Theorem, $K$ is compact in this topology. Furthermore, we have seen in Example \ref{ex1} that $D$ is dense, but not self segment-dense in $K.$

In this setting we define the vector-valued map
$$f:K\times K\to\R^2,\, f(x,y):=(\<x,y\>-1,\<x,y\>-1),$$
which has the following properties:

\begin{itemize}
\item[(a)] for all $y\in K,\, x\longrightarrow f(x,y)$ is continuous on $K,$

\item[(b)] for all $x\in K,\, y\longrightarrow f(x,y)$ is continuous on $K,$

\item[(c)] for all $x\in K,\, y\longrightarrow f(x,y)$ is affine, hence convex and also concave on $K,$

\item[(d)] $f(x,x)=(0,0)$ for all $x\in D.$
\end{itemize}

Further, consider the nonnegative orthant of $\R^2$, that is $\R^2_+=\{(x_1,x_2)\in\R^2:x_1\ge 0,\,x_2\ge 0\},$ which is obviously a convex and pointed cone.
We see that $f$ satisfies all the assumptions of Theorem \ref{t1},  except the assumption that $D$ is self segment-dense (here $D$ is only dense) and consequently the conclusion of the above mentioned theorem fails, since for $y=0\in K$ and for all $x\in K$, one has
$$f(x,y)=(-1,-1)\in -\inte \R^2_+.$$

\section{Self Segment-Dense Sets and the Strong Vector Equilibrium Problem}

Solution existence for the strong vector equilibrium problem (\ref{p2}), can be provided under some similar conditions as have been obtained in the  previous section for the weak vector equilibrium problem. However, note that the strong C-upper semicontinuity of a map, differs significantly from the C-upper semicontinuity property, as is shown in Remark \ref{difference} . Therefore, despite of similar statements to those presented in previous section, the results of this section can be considered original  and new. %differ from those obtained in the previous section.
For the sake of completeness we give some full proofs.

\begin{theorem}
\label{t2} Let $X$ and $Z$ be  Hausdorff, locally convex topological vector spaces,  let $C\subseteq Z$ be a convex,  pointed cone  and
let $K$ be a nonempty, convex and compact subset of $X$. Let $D\subseteq K$ be a
self segment-dense set and consider the mapping $f :K \times K \longrightarrow{Z}$  satisfying

\begin{itemize}
\item[(i)] $\forall y\in D,$ the mapping $x\longrightarrow f(x,y)$ is
strongly C-upper semicontinuous on $K,$

\item[(ii)] $\forall x\in K,$ the mapping $y\longrightarrow f(x, y)$ is strongly C-upper
semicontinuous on $K\setminus D$,

\item[(iii)] $\forall x\in D,$ the mapping $y\longrightarrow f(x,y)$ is C-convex on $%
D,$

\item[(iv)] $\forall x\in D,\, f(x, x)\not\in- C\setminus\{0\}.$
\end{itemize}

Then, there exists an element $x_0\in K$ such that
\begin{equation*}
f(x_0,y)\not\in -C\setminus\{0\},\,\forall y\in K.
\end{equation*}
\end{theorem}

\begin{proof}  Assume the contrary, that is, for every $x\in K$ there exists $y\in K$ such that $f(x,y)\in-C\setminus\{0\}.$ Then, for every $y\in K$ consider $V_y=\{x\in K: f(x,y)\in-C\setminus\{0\}\}.$ Obviously $K\subseteq \cup_{y\in K} V_y.$ We show that $(V_y)_{y\in D}$ is an open cover of $K.$ First of all observe that for all $y\in D,$ one has  $V_y=K\setminus G(y)$, where $G(y)$ is the set $\{x\in K:f(x,y)\not\in-C\setminus\{0\}\}.$ We show that $G(y)$ is closed for all $y\in D.$ Indeed, for fixed $y_0\in D$ consider the net $(x_\a)\subseteq G(y_0)$ and let $\lim x_\a=x_0.$ Assume that $x_0\not\in G(y_0).$ Then $f(x_0,y_0)\in-C\setminus\{0\}$.
 According to (i) the function $x\longrightarrow f(x,y_0)$ is strongly C-upper semicontinuous at $x_0$, hence  for any $k\in C\setminus\{0\}$, there exists a neighborhood $U$ of $x_0$ such that $f(u,y_0) \in f(x_0,y_0) + k - C\setminus\{0\}$ for all $u \in U.$ But then for $k=-f(x_0,y_0)\in C\setminus\{0\}$  one obtains that there exits $\a_0$ such that $f(x_\a,y_0)\in - C\setminus\{0\},$ for $\a\ge\a_0,$ which contradicts the fact that $(x_\a)\subseteq G(y_0)$. %for all $y\in D$, where $G$ is the set-valued map defined in the first part of the proof. But we have shown that $G(y)$ is closed for all $y\in D$,
Consequently, $V_y$ is open for every $y\in D.$

Assume now that there exists $x_0\in K$ such that $x_0\not\in \cup_{y\in D} V_y.$ Then $f(x_0,y)\not\in-C\setminus\{0\}$, for all $y\in D.$ We show that  $f(x_0,y)\not\in-C\setminus\{0\}$, for all $y\in K.$ Indeed, for $y_0\in K\setminus D,$ by the denseness of $D$ in $K,$ we have that  for every neighborhood $U$ of $y_0$ there exists a $u_0\in U\cap D.$  At this point, the assumption $(ii)$, the strongly C-upper semicontinuity of $f(x_0, y)$ on $K\setminus D$, assures that for all $k\in C\setminus\{0\}$ there exists a neighborhood $U$ of $y_0$ such that $ f(x_0,u)\in f(x_0,y_0)+k-C\setminus\{0\}.$ Assume that $f(x_0,y_0)\in -C\setminus\{0\}.$ Then let $k=-f(x_0,y_0)\in C\setminus\{0\}.$ Thus we have that  exists a neighborhood $U$ of $y_0$ such that $ f(x_0,U)\subseteq -C\setminus\{0\}.$ But, by choosing $u_0\in U\cap D$ we get that $f(x_0,u_0)\in -C\setminus\{0\},$ contradiction. Hence, $f(x_0,y)\not\in-C\setminus\{0\}$, for all $y\in K$, which contradicts our assumption, that for every $x\in K$ there exists $y\in K$ such that $f(x,y)\in-C\setminus\{0\}.$

Consequently, $(V_y)_{y\in D}$ is an open cover of the compact set $K$, in conclusion it contains a finite subcover. In other words, there exists $y_1,y_2,...,y_n\in D$ such that $K\subseteq \cup_{i=1}^n V_{y_i}.$ Consider $\big(p_i\big)_{i=\overline{1,n}}$  a continuous partition of  unity associated to the open cover $\big(V_{y_i}\big)_{i=\overline{1,n}}$. Then $p_i:K\To[0,1]$ is continuous and $\mbox{supp}(p_i)=\cl\{x\in K: p_i(x)\neq 0\}\subseteq V_{y_i}$ for all $i\in\{1,2,...,n\}$, moreover $\sum_{i=1}^n p_i(x)=1,$ for all $x\in K.$

Consider the mapping $\varphi:\co\{y_1,y_2,...,y_n\}\To\co\{y_1,y_2,...,y_n\}$
$$\varphi(x)=\sum_{i=1}^n p_i(x)y_i.$$
Obviously $\varphi$ is continuous, and $\co\{y_1,y_2,...,y_n\}$ is a compact and convex subset of the  finite dimensional  space $\mbox{span}\{y_1,y_2,...,y_n\}.$ Hence, by the Brouwer fixed point theorem, there exists $x_0\in \co\{y_1,y_2,...,y_n\}$ such that $\varphi(x_0)=x_0.$

Let $J=\{i\in\{1,2,...,n\}:p_i(x_0)>0\}.$ Obviously $J$ is nonempty, since $\sum_{i\in J} p_i(x_0)=1,$ and $$\varphi(x_0)=\sum_{i=1}^n p_i(x_0)y_i=\sum_{i\in J}p_i(x_0)y_i=x_0.$$
The latter  equality shows, that $x_0\in\co\{y_i:i\in J\}.$ On the other hand, from $p_i(x_0)>0$ for all $i\in J$ we obtain that $x_0\in\cap_{i\in J}V_{y_i}.$ Since $\cap_{i\in J}V_{y_i}$ is open, we  obtain  that there exists $U$ an open and convex neighbourhood of $x_0$ such that $U\subseteq \cap_{i\in J}V_{y_i}.$ Since $\co\{y_i:i\in J\}\cap U\neq\emptyset$, from Lemma \ref{l31}, we have that there exists $y_0\in \co\{y_i:i\in J\}\cap U\cap D.$
Hence, we  have $y_0=\sum_{i\in J}\l_i y_i\in \co\{y_i:i\in J\}\cap U\cap D$ for some $\l_i\ge0$ for all $i\in J$ and $\sum_{i\in J}\l_i=1$, and by (iv) in the hypothesis of the theorem one gets $f(y_0,y_0)\not\in-C\setminus\{0\}.$ On the other hand, by using (iii) we get $$f(y_0,y_0) =f(y_0,\sum_{i\in J}\l_i y_i)\le\sum_{i\in J}\l_i f(y_0,y_i),$$ which shows that $\sum_{i\in J}\l_i f(y_0,y_i)-f(y_0,y_0)\in C.$ But, $y_0\in U$, thus $f(y_0,y_i)\in-C\setminus\{0\}$ for all $i\in J.$ Hence $\sum_{i\in J}\l_i f(y_0,y_i)\in -C\setminus\{0\},$ which leads to $$-f(y_0,y_0)\in C-\sum_{i\in J}\l_i f(y_0,y_i)\subseteq C\setminus\{0\},$$ contradiction.

\end{proof}

\begin{remark} Obviously, one can give a proof of the previous theorem based on Ky Fan's Lemma, analogously to the proof of Theorem \ref{t1}. Therefore,  the rigid assumption of compactness of the set $K$ in the hypothesis of the previous theorem  can be replaced by its closedness and a coercivity condition. This can be done in the virtue of Fan's Lemma, which do not require the compactness of $G(y)$ for every $y\in K$ but in only one point. In conclusion, the following result holds.
\end{remark}

\begin{theorem}
\label{t21} Let $X$ and $Z$ be  Hausdorff, locally convex topological vector spaces,  let $C\subseteq Z$ be a convex,  pointed cone  and
let $K$ be a nonempty, convex and closed subset of $X$. Let $D\subseteq K$ be a
self segment-dense set and consider the mapping $f :K \times K \longrightarrow{Z}$  satisfying

\begin{itemize}
\item[(i)] $\forall y\in D,$ the mapping $x\longrightarrow f(x,y)$ is strongly
C-upper semicontinuous on $K,$

\item[(ii)] $\forall x\in K,$ the mapping $y\longrightarrow f(x, y)$ is strongly C-upper
semicontinuous on $K\setminus D$,

\item[(iii)] $\forall x\in D,$ the mapping $y\longrightarrow f(x,y)$ is C-convex on $%
D,$

\item[(iv)] $\forall x\in D,\, f(x, x)\not\in-  C\setminus\{0\},$

\item[(v)] $\exists K_0\subseteq X$ and $y_0\in D,$ such that $f(x,y_0)\in-C\setminus\{0\},$ for all $x\in K\setminus K_0.$
\end{itemize}

Then, there exists an element $x_0\in K$ such that
\begin{equation*}
f(x_0,y)\not\in -C\setminus\{0\},\,\forall y\in K.
\end{equation*}
\end{theorem}

\begin{proof}  The  proof is similar to the proof  of Theorem \ref{t11} therefore we omit it.
\end{proof}

\begin{remark} If $X$ is a reflexive Banach space, endowed with the weak topology, $\overline{B}_r:=\{x\in X:\|x\|\le r\}\subseteq X,$ is a closed ball with radius $r>0$, then one can take $K_0=K\cap \overline{B}_r$. Therefore,  condition (v) in Theorem \ref{t21} becomes:
$$\exists r>0\mbox{ and }y_0\in D,\mbox{ such that, for all }x\in K\mbox{ satisfying }\|x\|>r,$$
one has that
$$f(x,y_0)\in-C\setminus\{0\}.$$

Furthermore, in this setting, condition (v) in the hypothesis of Theorem \ref{t21} can be weakened by assuming that  there exists $r>0$ such that, for all $x\in K$ satisfying %
$\|x\|>r$, there exists some $y_0\in K$ with $\|y_0\|<\|x\|,$ and for which the  condition
$$f(x,y_0)\in-C\setminus\{0\}$$
holds.
\end{remark}

More precisely, we have the following result.

\begin{theorem}
\label{t22} Let $X$ be a reflexive Banach space and let $Z$ be  a Hausdorff, locally convex topological vector space,  let $C\subseteq Z$ be a convex,  pointed cone and let $K$ be a nonempty, convex and closed subset of $X$. Let $D\subseteq K$ be a
self segment-dense set in the weak topology of $X$, and consider the mapping $f :K \times K \longrightarrow{Z}$  satisfying

\begin{itemize}
\item[(i)] $\forall y\in D,$ the mapping $x\longrightarrow f(x,y)$ is strongly
C-upper semicontinuous on $K,$ in the weak topology of $X$,

\item[(ii)] $\forall x\in K,$ the mapping $y\longrightarrow f(x, y)$ is strongly C-upper
semicontinuous on $K\setminus D$, in the weak topology of $X$,

\item[(iii)] $\forall x\in K,$ the mapping $y\longrightarrow f(x,y)$ is C-convex on $K,$

\item[(iv)] $\forall x\in K,\, f(x, x)=0,$

\item[(v)] $\exists r>0$ such that, for all $x\in K$, $\|x\|>r$, there exists $y_0\in K$ with $\|y_0\|<\|x\|,$ such that
$f(x,y_0)\in- C.$
\end{itemize}

Then, there exists an element $x_0\in K$ such that
\begin{equation*}
f(x_0,y)\not\in -C\setminus\{0\},\,\forall y\in K.
\end{equation*}
\end{theorem}
\begin{proof} Let $r>0$ such that (v) holds, and let $r_1>r.$ Consider $K_0=K\cap \overline{B}_{r_1}.$ Obviously, $K_0$ is weakly compact, hence, according to Theorem \ref{t2} there exists $x_0\in K_0$ such that $f(x_0,y)\not\in-C\setminus\{0\},\,\forall y\in K_0.$ We claim that $f(x_0,y)\not\in-C\setminus\{0\},\,\forall y\in K.$

We show, as in the proof of Theorem \ref{t12}, that there exists $z_0\in K_0,\,\|z_0\|<r_1$ such that $f(x_0,z_0)=0.$ For $y\in K$ as in the proof of Theorem \ref{t12} we show that there exists $\l\in[0,1]$ such that $\l z_0 +(1-\l)y\in K_0,$  $f(x_0,\l z_0 +(1-\l)y)\not\in-C\setminus\{0\},$ and
$$(1-\l)f(x_0,y)-f(x_0,\l z_0 +(1-\l)y)\in C.$$
Assume that $f(x_0,y)\in-C\setminus\{0\}.$ Then, $-f(x_0,\l z_0 +(1-\l)y)\in -(1-\l)f(x_0,y)+ C\subseteq C\setminus\{0\},$ in other words
$$f(x_0,\l z_0 +(1-\l)y)\in-C\setminus\{0\},$$
contradiction.
\end{proof}

\begin{remark} One can observe, that that the conditions (iii) and (iv) in  the hypothesis of Theorem \ref{t22} considerably differ from the assumptions used in the hypothesis of Theorem \ref{t2}.  This is due the fact that condition (v) in Theorem \ref{t22} with the assumptions (iii) and (iv) of Theorem \ref{t2}  does not assure the existence of a solution for the strong vector equilibrium problem, when $K$ is  closed but not compact.

Next,  we obtain the existence of a solution for strong vector equilibrium problem under the original assumptions (iii) and (iv) by replacing (v) with a more suitable condition. In fact, we show that, if $\forall x\in K,\, y\longrightarrow f(x, y)$ is C-convex on $D,$ respectively $\forall x\in D,\,f(x,x)\not\in-C\setminus\{0\}$ instead of (iii), respectively (iv) in the previous theorem, then we can replace (v) with:

$\exists r>0,$ such that, for all $x\in K$ satisfying $ \|x\|\le r,$ there exists $y_0\in D$ with $\|y_0\|<r,$  such that $f(x,y_0)\in-C.$
\end{remark}

The following result holds.
\begin{theorem}
\label{t23} Let $X$ be a reflexive Banach space and let $Z$ be  a Hausdorff, locally convex topological vector space,  let $C\subseteq Z$ be a convex,  pointed cone  and let $K$ be a nonempty, convex and closed subset of $X$. Let $D\subseteq K$ be a
self segment-dense set in the weak topology of $X$, and consider the mapping $f :K \times K \longrightarrow{Z}$  satisfying

\begin{itemize}
\item[(i)] $\forall y\in D,$ the mapping $x\longrightarrow f(x,y)$ is strongly
C-upper semicontinuous on $K,$ in the weak topology of $X$,

\item[(ii)] $\forall x\in K,$ the mapping $y\longrightarrow f(x, y)$ is strongly C-upper
semicontinuous on $K\setminus D$, in the weak topology of $X$,

\item[(iii)] $\forall x\in K,$ the mapping $y\longrightarrow f(x,y)$ is C-convex on $D,$

\item[(iv)] $\forall x\in D,\, f(x, x)\not\in-C\setminus\{0\},$

\item[(v)] $\exists r>0$ such that, for all $x\in K$, $\|x\|\le r$, there exists $y_0\in D$ with $\|y_0\|<r,$ such that
$f(x,y_0)\in-C.$
\end{itemize}

Then, there exists an element $x_0\in K$ such that
\begin{equation*}
f(x_0,y)\not\in -C\setminus\{0\},\,\forall y\in K.
\end{equation*}
\end{theorem}
\begin{proof}  For the sake of completeness we give a full proof. Let $r>0$ such that (v) holds, and consider $K_0=K\cap \overline{B}_{r}.$ Obviously, $K_0$ is weakly compact, hence, according to Theorem \ref{t2} there exists $x_0\in K_0$ such that $f(x_0,y)\not\in-C\setminus\{0\},\,\forall y\in K_0.$ We show, that $f(x_0,y)\not\in-C\setminus\{0\}\,\forall y\in K.$
According to (v) there exists $z_0\in D$ with $\|z_0\|<r,$ such that $f(x,z_0)\in-C.$ On the other hand, $z_0\in K_0,$ hence  $f(x_0,z_0)\not\in-C\setminus\{0\}.$ Consequently $f(x_0,z_0)=0.$
Let $y\in D\setminus K_0.$ Then, in virtue of self-segment-denseness of $D$ in $K$, there exists $\l\in[0,1]$ such that $\l z_0 +(1-\l)y\in D\cap K_0,$ consequently $f(x_0,\l z_0 +(1-\l)y)\not\in-C\setminus\{0\}.$ From (iii) we have
$$\l f(x_0,z_0)+(1-\l)f(x_0,y)-f(x_0,\l z_0 +(1-\l)y)\in C$$
or, equivalently
$$(1-\l)f(x_0,y)-f(x_0,\l z_0 +(1-\l)y)\in C.$$
Assume that $f(x_0,y)\in-C\setminus\{0\}.$ Then, $-f(x_0,\l z_0 +(1-\l)y)\in -(1-\l)f(x_0,y)+ C\subseteq C\setminus\{0\},$ in other words
$$f(x_0,\l z_0 +(1-\l)y)\in-C\setminus\{0\},$$
contradiction.

Hence, $f(x_0,y)\not\in-C\setminus\{0\},$ for all $y\in D.$
Finally, if $y\in K\setminus D$ by the denseness of $D$ in $K$  for every neighborhood $U$ of $y$ there exists a $u_0\in U\cap D.$  At this point, the assumption $(ii)$, the proper C-upper semicontinuity of $y\To f(x_0, y)$ on $K\setminus D$, assures that for all $k\in C\setminus\{0\}$ there exists a neighborhood $U$ of $y$ such that $ f(x_0,u)\in f(x_0,y)+k-C\setminus\{0\}.$ Assume that $f(x_0,y)\in -C\setminus\{0\}.$ Then let $k=-f(x_0,y)\in C\setminus\{0\}.$ Thus we have that  exists a neighborhood $U$ of $y$ such that $ f(x_0,u)\in -C\setminus\{0\}.$ But, by choosing $u_0\in U\cap D$ we get that $f(x_0,u_0)\in -C\setminus\{0\},$ contradiction.
\end{proof}

In what follows, we show that the assumption that $D$ is self-segment-dense, in the hypotheses %
of the previous theorems, is essential and it cannot be replaced by the denseness of $D.$ Indeed, let us consider the Hilbert space of square-summable sequences $l_2$, and let $K=\{x\in l_2:\|x\|\le 1\}$ be its unit ball, while $D=\{x\in l_2:\|x\|= 1\}$ is the unit sphere. It is well known that $l_2,$ endowed with the weak topology, is a Hausdorff locally convex topological vector space, and by Banach-Alaoglu Theorem, $K$ is compact in this topology. Furthermore, we have seen in Example \ref{ex1} that $D$ is dense, but not self segment-dense in $K.$

In this setting we define the vector-valued map
$$f:K\times K\To\R^2,\, f(x,y):=(\<x,y\>-1,0).$$ Further, consider the $C=\R_+\times \{0\}=\{(x,0): x\in\R,\,x\ge 0\},$ which is obviously a convex pointed cone.
It can easily be verified that  all the assumptions of  Theorem \ref{t2} are satified, except the assumption that $D$ is self segment-dense (here $D$ is only dense) and also that the conclusions of the above mentioned theorem fails, since for $y=0\in K$ and for all $x\in K$, one has
$$f(x,y)=(-1,0)\in  -C\setminus\{(0,0)\}.$$

%%%%%%%%%%%%%%%%%%%%%%%%%%%%%%%%%%%%%%%%%%%%%%%%%%%%%%%%%%%%%%%%%%%

\section{Applications}

In this section, we apply the results obtained to vector optimization problems and vector variational inequalities.
\subsection{Vector Optimization}

Let $X$ and $Z$ be  Hausdorff, locally convex topological vector spaces,  let $C\subseteq Z$ be a convex,  pointed cone  and
let $K$ be a nonempty subset of $X$.
Let $F:K\To Z$ be a vector function. Recall that $x_0\in K$ is called a weakly  efficient point, respectively efficient point  of $F$ \cite{GGR,CM}, iff $F(y)-F(x_0)\not\in-\inte C$ for all $y\in K,$ respectively $F(y)-F(x_0)\not\in- C\setminus\{0\}$ for all $y\in K.$
Based on the results obtained in the  previous sections, we  can  state some results regarding the existence of a weakly efficient point, respectively efficient point of a vector function $F$. Consider the mapping  $f:K\times K\To Z,\, f(x,y)=F(y)-F(x).$  It is obvious that $x_0\in K$ is a weak efficient point, respectively efficient point of $F$, if and only if, $x_0$ is a solution of the weak vector equilibrium problem $f(x_0,y)\not\in -\inte C,\,\forall y\in K,$ respectively $x_0$ is a solution of the strong vector equilibrium problem $f(x_0,y)\not\in-C\setminus\{0\},\,\forall y\in  K.$

Assume now that $K$ is a nonempty, convex and compact subset of $X$. Let $D\subseteq K$ be a
self segment-dense set.
It can easily be verified that in this case, the assumptions (i) in the hypothesis of Theorem \ref{t1} becomes
 $F$ is C-lower semicontinuous on $K.$ Then, since $K$ is compact, one can conclude that there exists a  weak  efficient point of $F,$ see \cite{DTL1}. Note that the condition (iv) in the hypothesis of Theorem \ref{t1} is automatically satisfied. Moreover, it is obvious that the condition (iii) in the hypothesis of Theorem \ref{t1},  which becomes that  $F$ is C-convex on $D,$ is superfluous. Nevertheless, this is not the case when $K$ is not a compact set, but a closed subset of a reflexive Banach space $X.$
In conclusion, according to the Theorem \ref{t13},  in reflexive Banach spaces the following result holds.

\begin{theorem} Let $X$ be a reflexive Banach space and let $Z$ be  a Hausdorff, locally convex topological vector space,  let $C\subseteq Z$ be a convex,  pointed cone with nonempty interior and let $K$ be a nonempty, convex and closed subset of $X$. Let $D\subseteq K$ be a
self segment-dense set in the weak topology of $X$, and consider the mapping $F :K  \longrightarrow{Z}$  satisfying

\begin{itemize}
\item[(i)] $F$ is  C-lower semicontinuous on $K,$
\item[(ii)] $F$ is  C-upper semicontinuous on $K\setminus D,$
\item[(iii)] $F$ is C-convex on $D,$
\item[(iv)] $\exists r>0$ such that, for all $x\in K$, $\|x\|\le r$, there exists $y_0\in D$ with $\|y_0\|<r,$ such that
 $F(y_0)-F(x)\in-\inte C\cup \{0\}.$
\end{itemize}
Then, there exists a weak  efficient point of $F.$
\end{theorem}
\begin{proof} The conclusion follows from  Theorem \ref{t13} applied to the mapping $f:K\times K\To Z,\, f(x,y)=F(y)-F(x).$
\end{proof}

We show next that assumption (iii) in the hypothesis of the  previous theorem, is essential.
\begin{example} Let $f:[0,\infty)\To \R,\, f(x)=\left\{\begin{array}{ll} 0,\mbox{ if } x\le 1\\-x+1,\mbox{ if }x>1.
\end{array}\right.$ Consider the convex cone $C=[0,\infty).$ Let $D$ be the set of non-negative rational numbers. Then obviously, (i) and (ii) are satisfied automatically and (iv) is satisfied with $r=1.$ It is also obvious that $f$ is not convex on $D$. Hence, (iii) fails, and also the conclusion of the previous theorem, since $f$ has no minima.
\end{example}

Similarly, Theorem \ref{t23}, applied to the bifunction $f:K\times K\To Z,\, f(x,y)=F(y)-F(x),$ assures the existence of an efficient point in reflexive Banach spaces. More precisely, the following result holds.

\begin{theorem} Let $X$ be a reflexive Banach space and let $Z$ be  a Hausdorff, locally convex topological vector space,  let $C\subseteq Z$ be a convex,  pointed cone with nonempty interior and let $K$ be a nonempty, convex and closed subset of $X$. Let $D\subseteq K$ be a
self segment-dense set in the weak topology of $X$, and consider the mapping $F :K  \longrightarrow{Z}$  satisfying

\begin{itemize}
\item[(i)] $F$ is strongly C-lower semicontinuous on $K,$
\item[(ii)] $F$ is strongly C-upper semicontinuous on $K\setminus D,$
\item[(iii)] $F$ is C-convex on $D,$
\item[(iv)] $\exists r>0$ such that, for all $x\in K$, $\|x\|\le r$, there exists $y_0\in D$ with $\|y_0\|<r,$ such that
$F(y_0)-F(x)\in -C$.
\end{itemize}
Then, there exists an efficient point of $F.$
\end{theorem}

\subsection{Minty Type Vector Variational Inequalities}
In this section we give some new existence results for weak, respectively strong vector variational
inequalities of Minty type.
Let $X$ and $Z$ be  Hausdorff, locally convex topological vector spaces,  let $C\subseteq Z$ be a convex,  pointed cone  and
let $K$ be a nonempty subset of $X$. Let us denote $L(X,Z)$ the set of all linear and continuous operators
from $X$ to $Z$. Let $F:X\To L(X,Z).$ For $x^*\in L(X,Z)$ and $x\in X$, we denote by $\<x^*,x\>$ the vector $x^*(x)\in Z.$ The weak  vector variational inequality of Minty type reads

find $x_0\in K$ such that $\<F(y),y-x_0\>\not\in-\inte C,$ for all $y\in K$.
\\
Analogously, the strong  vector variational inequality of Minty type reads

find $x_0\in K$ such that $\<F(y),y-x_0\>\not\in- C\setminus\{0\},$ for all $y\in K$.

Note that if we take $f:K\times K\To Z,\,f(x,y)=\<F(y),y-x_0\>$ the weak, respectively strong vector variational inequality problem of Minty type becomes the appropriate vector equilibrium problem. In conclusion the following results hold.

\begin{theorem} Let $X$ and $Z$ be  Hausdorff, locally convex topological vector spaces,  let $C\subseteq Z$ be a convex,  pointed cone with nonempty interior and let $K$ be a nonempty, convex and compact subset of $X$. Let $D\subseteq K$ be a
self segment-dense set and consider the mapping $F:K\To L(X,Z)$ satisfying

\begin{itemize}
%\item[(i)] $\forall y\in D$ the mapping $x\longrightarrow \<F(y),y-x\>)$ is
%C-upper semicontinuous on $K,$
\item[(i)] $\forall x\in K$ the map $y\longrightarrow \<F(y),y-x\>$ is C-upper semicontinuous on $K\setminus D$,
\item[(ii)] $\forall x\in D,$ the mapping $y\longrightarrow \<F(y),y-x\>$ is C-convex on $%
D,$
\end{itemize}

Then, there exists an element $x_0\in K$ such that
\begin{equation*}
\<F(y),y-x_0\>\not\in -\inte C,\,\forall y\in K.
\end{equation*}
\end{theorem}
\begin{proof} The conclusion follows by Theorem \ref{t1} applied to the mapping $f:K\times K\To Z,\,f(x,y)=\<F(y),y-x_0\>$.
\end{proof}

\begin{remark}\label{difference} Note that the condition (i) of Theorem \ref{t1}, that is, $x\longrightarrow \<F(y),y-x\>$ is C-upper semicontinuous on $K$ for every $y\in D$ is automatically satisfied. However this is not the case for strongly C-upper semicontinuity. Indeed, let $X=Z=\R^2$ and
$$F:X\to L(X,Z),\, F(u,v)=\left(                                                                                                                 \begin{array}{cc}
u & 0 \\
 0 & v \\
 \end{array}
 \right).$$
 Then $\<F(u,v),(u,v)-(x,y)\>=(u(u-x),v(v-y)).$ Let $C$ be the non-negative orthant of $\R^2$, i.e. $C=\R^2_+.$
Let $K=[-1,1]\times[-1,1]$ and $D=\mathbb{Q}^2\cap K,$ where $\mathbb{Q}$ is the set of rational numbers. According to Example \ref{EEE}, $D$ is self-segment-dense in $K.$ We show that for $(u,v)=(-1,-1)\in D$ the map  $(x,y)\longrightarrow (u(u-x),v(v-y))$ is not strongly C-upper semicontinuous at $(x,y)=(0,1).$ Indeed, assume the contrary, that is, for all $(k,h)\in C\setminus \{(0,0)\}$ there exits $U$ a neighbourhood of $(0,1)$, such that for all $(s,t)\in U$ one has $(s+1,t+1)\in (1,2)+(k,h)-\R^2_+\{(0,0)\}.$ Obviously one can take $U=(-\e,\e)\times(1-\e,1+\e)$ for some $\e>0.$  Let $(k,h)=(0,1).$ Then, one must have $(s,t)\in (0,2)-\R^2_+\setminus\{(0,0)\}$ for all $(s,t)\in U$, which leads to contradiction if one takes $s>0.$
\end{remark}
In conclusion, as an application of Theorem \ref{t2}, for strong vector variational inequalities the following result holds.

\begin{theorem} Let $X$ and $Z$ be  Hausdorff, locally convex topological vector spaces,  let $C\subseteq Z$ be a convex,  pointed cone  and
let $K$ be a nonempty, convex and compact subset of $X$. Let $D\subseteq K$ be a
self segment-dense set and consider the mapping $F:K\To L(X,Z)$ satisfying

\begin{itemize}
\item[(i)] $\forall y\in D$ the mapping $x\longrightarrow \<F(y),y-x\>$ is
strongly C-upper semicontinuous on $K,$

\item[(ii)] $\forall x\in K,\, y\longrightarrow \<F(y),y-x\>$ is strongly C-upper
semicontinuous on $K\setminus D$,

\item[(iii)] $\forall x\in D,$ the mapping $y\longrightarrow \<F(y),y-x\>$ is C-convex on $%
D$
\end{itemize}

Then, there exists an element $x_0\in K$ such that
\begin{equation*}
\<F(y),y-x_0\>\not\in -C\setminus\{0\},\,\forall y\in K.
\end{equation*}
\end{theorem}

\begin{remark} One can use Theorem \ref{t13}, respective Theorem \ref{t23} to obtain some sufficient conditions that ensure the solution existence of weak vector variational inequalities of Minty type, respectively strong vector variational inequalities of Minty type, in reflexive Banach spaces.
\end{remark}

%\subsection{Conic Saddle Points}

{\bf Acknowledgements} This work was supported by a grant of the Romanian Ministry of Education, CNCS - UEFISCDI, project number PN-II-RU-PD-2012-3 -0166.

\end{document}